\documentclass[12 pt]{amsart}

\usepackage{amsmath}
\usepackage{amssymb}
\usepackage{amsfonts}

\usepackage[largesc]{newpxtext}
\usepackage[varbb]{newpxmath}
\usepackage{float}

\usepackage[margin = 1 in]{geometry}
\usepackage{microtype}

\usepackage{amsthm, thmtools, thm-restate}
\newtheorem{theorem}{Theorem}

\newtheorem{lemma}[theorem]{Lemma}

\theoremstyle{definition} % definition style
\newtheorem{definition}[theorem]{Definition}

\theoremstyle{remark} % remark style

\usepackage{hyperref}
\usepackage[capitalise]{cleveref}
\hypersetup{ % make links pretty and not boxed
colorlinks,
linkcolor = blue,
urlcolor = blue,
citecolor = blue}

\usepackage[textsize = footnotesize, disable]{todonotes}
\todostyle{jayden}{color = pink}
\todostyle{jasper}{color = lime}
\todostyle{alan}{color = purple}

\usepackage{transparent}
\usepackage{tikz}
\usetikzlibrary {arrows.meta}

\newcommand{\mzncarrot}[4][(0,0)]{
    \begin{scope}[shift={#1}]
    \shade[top color=#2, bottom color=#2, shading angle=90, draw=#2!90!black, rounded corners=0.5ex, thick] (0.5,0.25) -- ++(0,1.5+.25) -- ++(0.5,0.25) -- ++(1.5,0) -- ++(1,-.5) -- ++(0,-1.45) -- cycle;
    \draw[thick, rounded corners=0.2ex, fill=white, thick] (1,2.1) -- ++(0.5,0.45) -- ++(0.8,0) -- ++(0.3,-0.45) -- cycle;
    \draw[thick] (2.1-.25,2.1) -- ++(0,.45);
    \if\relax\detokenize{#3}\relax 
    \else
    \draw[fill=#2!20, thin, draw = #2] (2.1-.25,1.3) circle[radius=.65] 
        node [anchor = center, rotate = #4]{\small #3};
    \fi
    \draw[fill=gray!80,thin] (1.375-.25,.25) circle[radius=.2];
    \draw[fill=gray!80,thin] (2.76-.25,.25) circle[radius=.2];
    \end{scope}
}

\newcommand{\mzncar}[3][(0,0)]{
    \mzncarrot[#1]{#2}{#3}{0}
}

\newcommand{\tikspotc}[3][(0,0)]{
    \begin{scope}[shift={#1}]
        \draw[fill = #3, rounded corners = .3ex, draw = #3] (0,-.4) rectangle node[text = white]{#2} ++(4, -1.2);
        \draw[-to, white,very thick] (.5,-.4-.6)--++(.001,0);
    \end{scope}
}

\newcommand{\tikspot}[2][(0,0)]{
    \tikspotc[#1]{#2}{black}
}

\newcommand{\tikpref}[3][0,0]{
    \begin{scope}[shift={#1}]
        \draw[fill=#2!20, draw = #2] (0,0) circle[radius=.65] node [anchor = center]{\small #3};
    \end{scope}
}

\title{Preference-restricted parking functions}

\author[Bown]{Jasper Bown}
\address[J. Bown]{University of Minnesota --- Twin Cities, United States}
\email{bown0009@umn.edu}

\author[Kagey]{Peter Kagey}
\address[P. Kagey]{California State Polytechnic University, Pomona}
\email{pkagey@cpp.edu}

\author[Kappler]{Alan Kappler}
\address[A. Kappler]{Harvey Mudd College, United States}
\email{akappler@g.hmc.edu}

\author[Orrison]{Michael E.~Orrison}
\address[M.~Orrison]{Harvey Mudd College, United States}
\email{orrison@hmc.edu}

\author[Thadani]{Jayden Thadani}
\address[J. Thadani]{Harvey Mudd College, United States}
\email{jthadani@hmc.edu}

\begin{document}

\begin{abstract}
     A parking function is a function $\pi:[n]\to [n]$ whose $i$th-smallest output is at most $i,$ corresponding to a parking procedure for $n$ cars on a one-way street. We refine this concept by introducing preference-restricted parking functions, which are parking functions with codomain restricted to some $S\subseteq[n]$. Particular choices of $S$ yield new combinatorial interpretations of previous results about variant parking procedures, and new results too. In particular we consider prime parking functions, parking procedures with fewer spots than cars, and parking functions where each spot has space for multiple cars. We also use restricted parking functions to reprove Abel's binomial theorem.
\end{abstract}

\maketitle

\section{Introduction} \label{sec:intro}

Take $n$ cars for some $n\in\mathbb{Z}_{> 0},$ all of which are attempting to park along a one-way street with a total of $n$ spots. Each car has a preferred spot, described by a \emph{preference function} $\pi:[n]\to[n]$, or equivalently, a \emph{preference list} $(\pi(1), \dots, \pi(n))$ of length $n$, where the $i$th car prefers the $\pi(i)$th spot. The cars attempt to park along the one-way street as follows. Starting with the first car, each car successively enters the street and goes to its preferred spot. If that spot is empty, the car parks there; otherwise it proceeds down the street and parks in the first empty spot available. If it reaches the $n$th spot and finds nowhere to park, it leaves without parking.

Some choices of $\pi$ lead to all $n$ cars parking. For example, any bijection $\pi:[n]\to[n]$ lets each car park in its preferred spot. Also, if every car prefers the first spot (if $\pi$ is constant at $1$) each $i$th car parks in the $i$th spot. However, not all functions $\pi$ allow all cars to park. If no car prefers the first spot, then the first spot will never be parked in and not all $n$ cars can park in the remaining $n-1$ spots. The \emph{parking functions} on $n$ cars are those $\pi : [n] \to [n]$  which allow all cars to park.

Parking functions are well-studied; we suggest \cite{yan-survey-2015} as a comprehensive introduction to their properties. For instance, they have been enumerated and characterized precisely. There are $(n + 1)^{n - 1}$ parking functions on $n$ cars. They are exactly the functions $[n] \to [n]$ satisfying the following \emph{Catalan condition}. A function $\pi : [n] \to [n]$ is a parking function if and only if $\#\pi^{-1}([i]) \ge i$ for each $i \in [n]$. That is, a function is a parking function if at least $i$ cars prefer to park in one of the first $i$ spots. This property can be used to define parking functions.

There are many, many possible variants of the problem of enumerating and analyzing parking functions, some of which have gained traction in recent years. The interested reader is encouraged to peruse \cite{carlson-harris-2020} for an engaging overview of such variants. In many cases, these variants seem sufficiently different from classical parking functions that the developed body of theory cannot be applied immediately. We might ask:
\begin{enumerate}
	\item What if the inequality in the Catalan condition is required to be strict? One obtains the building blocks of parking functions. These are \emph{prime parking functions}, of independent interest. See \cite{meyles-2023} for an excellent overview of their properties.
 	\item What if there are fewer cars than spots? Which functions minimize the number of cars unable to park? (This is a special case of the variants considered in \cite{cameron-johannsen-prellberg-schweitzer-2008}.)
	\item What if each parking spot can accommodate more than one car? (Perhaps, as in \cite{blake-konheim-1977}, our parking spots are hash buckets and our cars are data, or as in \cite{carlson-harris-2020}, our parking spots are clown cars between which different clowns have different preferences.)
\end{enumerate}
We will characterize and enumerate these by placing each in natural bijection with our new notion of \emph{$S$-restricted parking functions} for a specific choices of $S \subseteq [n]$.

\begin{definition}
    Given a subset of allowed preferences $S \subseteq [n]$, an \emph{$S$-restricted parking function} on $n$ cars is a function $\pi : [n] \to S$ which is also a parking function. 

    We denote the set of $S$-restricted parking functions by $\mathrm{PF}_{n \mid S}$ and their number by $\# \mathrm{PF}_{n \mid S}$.
\end{definition}

These reinterpretations yield new combinatorial insights, simpler proofs, more complete theories of parking procedures, and connections to deep results in combinatorics. In some sense, $S$-restricted parking functions provide a paradigm for solving parking questions that strikes the balance between the generality of $\mathbf{u}$-parking functions (as defined in \cite{yan-survey-2015} or \cref{sec:connections}) and the combinatorial intuition of the original parking functions. \todo[jayden]{should we put theorem statements here?}

\subsection{An aside on the Catalan condition}

The Catalan condition can also be written in the language of order-preserving parking functions. This will be useful in proofs.

For each parking function, there is a unique \emph{order-preserving rearrangement}. It is a rearrangement $\pi^{\uparrow} = \pi \circ \sigma$ (for some permutation $\sigma$) such that $\pi^\uparrow(i)\le \pi^\uparrow(j)$ if and only if $i\le j$. The corresponding list $(\pi^{\uparrow}(1), \dots, \pi^{\uparrow}(n))$ is non-decreasing, and for this reason $\pi^{\uparrow}$ is called a \emph{non-decreasing parking function}. 

The Catalan condition can equivalently be stated: $\pi$ is a parking function if and only if $\pi^\uparrow(i) \le i$ for all $i \in [n]$. The equivalent conditions are called Catalan because the non-decreasing parking functions are counted by Catalan numbers and are in bijection with Catalan objects like Dyck paths.

\section{Prime parking functions} \label{primesection}

Prime parking functions are an important restricted class of parking functions. In analogy with prime numbers, they are the indecomposable parts from which all parking functions are built. They are often defined as satisfying a stricter Catalan condition than parking functions:

\begin{definition}
	A \emph{prime parking function} on $n$ cars is a function $\pi : [n] \to [n],$ such that for all $i\in[n-1]$ the size of $\pi^{-1}([i])$ is greater than $i$. Equivalently, it is a function whose order-preserving rearrangement satisfies $ \pi^{\uparrow}(i) < i$ for each $1<i\le n$. 

	We denote the set of all prime parking functions on $n$ cars by $\mathrm{PPF}_{n}$ and their number by $\# \mathrm{P P F}_{n}$.
\end{definition}

A prime parking function, then, describes the preferences of $n$ cars on a street with $n$ cars such that for all $i < n$, more than $i$ cars prefer one of the first $i$ spots. This means that for each parking spot (except the last), there is a car that attempts to park in the spot after it is occupied.

Prime parking functions were first introduced by Gessel in private correspondences mentioned both in Stanley’s Enumerative combinatorics \cite{stanley-fomin-1999} and in Kalikow’s thesis \cite{kalikow-1999}. Recently, prime parking functions have served as a useful tool to study unit interval parking functions \cite{meyles-2023} and have been generalized to different kinds of parking functions \cite{armon-2024}.

By analogy to normal $S$-restricted parking functions, we may consider $S$-restricted prime parking functions as well:

\begin{definition}
    An \emph{$S$-restricted prime parking function} on $n$ cars is a prime parking function on $n$ cars with image in $S$.

    We denote the set of all $S$-restricted prime parking functions on $n$ cars by $\mathrm{PPF}_{n \mid S}$ and their number by $\# \mathrm{P P F}_{n \mid S}$.
\end{definition}

Prime parking functions themselves have a characterization in terms of $S$-restricted parking functions: prime parking functions are in bijection with parking functions where no car prefers the second spot. This bijection is given by moving the preference of each ``extra car'' preferring one of the first $i$ spots one spot over.

\begin{restatable}{theorem}{primeIsRes}
    \label{thm:primeIsRes}
	There is a bijection between prime parking functions and $[n] \setminus \{ 2 \}$-restricted parking functions (on $n$ cars). 
\end{restatable}

This is a corollary of \cref{thm:resPrimeIsRes}, which represents $S$-restricted prime parking functions as $T$-restricted parking functions for some related subset $T$. It is given by the same kind of bijection described above.

\begin{restatable}{theorem}{resPrimeIsRes}
    \label{thm:resPrimeIsRes}
	For any $S\subseteq [n]$ such that $1\in S,$  there is a bijection between $S$-restricted prime parking functions and $T$-restricted parking functions (on $n$ cars) where
	\[
		T = \{ 1 \} \cup \{ i + 1 \mid i \in S, 1 < i < n \}.
	\]
	That is,
	\[
		\# \mathrm{PPF}_{n \mid S} = \# \mathrm{PF}_{n \mid T}.
	\]
\end{restatable}

\begin{proof}
	Consider the bijection $f \colon S \setminus \{ n \} \to T$ by $f(1) = 1$ and $f(x) = x + 1$ otherwise. We claim $f$ pushes forward to a bijection from $S$-restricted prime parking functions to $T$-restricted parking functions by $f_{*}\colon\mathrm{PPF}_{n \mid S}\to\mathrm{PF}_{n \mid T}$ defined by $\pi \mapsto f_{*} \pi = f \circ \pi$.

	The pushforward $f_*$ is a well-defined map from functions $\mathrm{PPF}_{n \mid S}$ to functions $[n]\to T$. In particular, no prime parking function $\pi$ can have $n$ in its image (since $\pi$ is prime, at least $n > n - 1$ cars prefer to park in the first $n - 1$ spots). Thus, $\pi$ has image contained in $S \setminus \{ n \}$, and $f_{*} \pi$ then must have image in $f(S \setminus \{ n \}) = T$. 

	Now we show that every $f_{*} \pi$ is a parking function. Since $\pi$ is prime, there are two cars ``$j$'' such that $\pi(j) = 1$. Since $f(1) = 1$, the same two cars have $f^{*} \pi(j) = 1$. Thus, the Catalan condition is satisfied for $i = 1$ (and $i = 2$). Now suppose $i \ge 2$. Since $\pi$ is prime, at least $i$-many cars ``$j$'' satisfy $\pi(j) \le i - 1$. Since $f(x) \le x + 1$, these same $i$-many cars have $f_{*} \pi(j)\le i$. That is, the Catalan condition is satisfied for all $i$ and $f_{*} \pi$ is a parking function. It has image in $T$, so $f_{*} \pi \in \mathrm{PF}_{n\mid T}$.

	The function $f\colon S\backslash\{n\}\to T$ is invertible, inducing an inverse map to the pushforward. In particular we claim there is a map $f^{*} : \mathrm{PF}_{n \mid T} \to \mathrm{PPF}_{n \mid S}$ by $\psi \mapsto f^{*} \psi = f^{-1} \circ \psi$ and that $f^{*}$ is the inverse to $f_{*}$. $f^{*}$ is a well-defined map sending $\mathrm{PF}_{n \mid T}$ to functions $[n] \to S$ for the same reason that $f_{*}$ is well-defined. Specifically, $\psi \in \mathrm{PF}_{n \mid T}$ has image in $T$, so $f^{*} \psi$ has image in $f^{-1}(T) = S \setminus \{ n \}$. 

	It remains to show that $f^{*} \psi$ is a prime parking function. The argument is similar to that by which $f_{*} \pi$ was shown to be a parking function. Note that $f^{-1}(1) = 1$ and $f^{-1}(x) = x - 1$ for $x \neq 1$. Since $2 \not \in T$ (to satisfy the Catalan condition when $i = 2$) there are at least two cars ``$j$'' such that $\psi(j) = 1$ and so $f^{*} \psi(j) = 1$. For every $1 < i < n$, there are at least $(i + 1)$-many cars ``$j$'' such that $\psi(j) \le i + 1$ and thus, more than $i$-many cars such that $f^{*}\psi(j) \le i$. Therefore, the strict Catalan condition is satisfied for all $i$ and $f^{*} \psi \in \mathrm{PPF}_{n\mid S}$.

	Now we can compose $f_{*}$ and $f^{*}$ to see that $f^{*}(f_{*} \pi) = f^{-1} \circ f \circ \pi = \pi$ and similarly, $f_{*}(f^{*} \psi) = f \circ f^{-1} \circ \psi = \psi$. Thus, $f_*:\mathrm{PPF}_{n\mid S}\to\mathrm{PF}_{n\mid T}$ is a bijection and $\# \mathrm{PPF}_{n \mid S} = \# \mathrm{PF}_{n \mid T}$.
\end{proof}

\section{Initial segment restrictions}
 
One variant on classical parking functions is when the number of cars differs from the number of spots to park in. \cite{cameron-johannsen-prellberg-schweitzer-2008} studies this idea in many variations. We are interested in the case where there are more cars than parking spots. Not all cars will be able to park, but we can try to minimize the \textit{defect} of the function --- the number of cars which are unable to park. This problem of minimizing defect for preference functions of $n$ cars on $s$ spots is equivalent to finding $S$-restricted parking functions for a particular $S \subseteq [n]$.

\begin{theorem}
    Given $1\le s\le n$, the preference functions $\psi : [n] \to [s]$ with the minimum possible defect $n - s$ are in bijection with the $[s]$-restricted parking functions $\mathrm{PF}_{n \mid [s]}$ (where $[s] = \{ 1, 2, \dots, s \}$).
\end{theorem}

\begin{proof}
	Suppose $\pi$ is an $[s]$-restricted parking function. Then $\pi : [n] \to [n]$ defines a unique function $\psi : [n] \to [s]$ with $\psi(i) = \pi(i)$ for each $i$. We know that performing the parking procedure with preferences defined by $\pi$, and thus also by $\psi$, leads to $s$ cars parking in the first $s$ spots. Thus, under the parking procedure for $\psi$, $s$ cars park in the first $s$ spots and exactly $n - s$ cars are unable to park.

	Suppose $\psi : [n] \to [s]$ has defect $n - s$. $\psi$ defines a unique function $\pi : [n] \to [n]$ by composing with the inclusion $[s] \to [n]$. We claim that $\pi$ satisfies the Catalan condition, and thus, $\pi \in \mathrm{PF}_{n \mid [s]}$. Since $\pi(i) = \psi(i)$ by its definition, this is sufficient to show that composing with the inclusion is the inverse to considering  $\pi \in \mathrm{PF}_{n \mid [s]}$ as a defect-minimizing function $[n] \to [s]$, and completes the proof of the bijection.

	Since all of the first $s$ spots are filled, at least $i$ cars prefer the first $i$ spots for all $i \le s$ (this can be seen explicitly by the fact that $\pi$ ``contains a parking function'' when restricted to the $T$ cars parking in the first $s$ spots). Since all cars park in the first $[s]$ spots, for each $i > s$ there are $n$ (and thus, at least $i$) cars preferring to park in the first $i$ spots.
\end{proof}

We can ask ``how many $[s]$-restricted parking functions are there?'' This is answered in \cite{cameron-johannsen-prellberg-schweitzer-2008} through a generating function argument (Definition 5, Theorem 6). They also provide an alternate combinatorial proof which leads to a recursive expression rather than an explicit one. With a modification, we present a combinatorial argument which gives an explicit expression for the number of $[s]$-restricted parking functions.

\begin{restatable}{theorem}{resPFcount1}
    \label{thm:resPFcount1}
    Let $1\le s\le n.$ Then the number of $[s]$-restricted parking functions on $n$ cars is
    \[\#\mathrm{PF}_{n\mid [s]}=s^{n} - \sum_{i = 0}^{s - 1} \binom{n}{i} (i + 1)^{i - 1} (s - i - 1)^{n - i}.\]
\end{restatable}

\begin{proof}
    We count all non-parking functions $[n]\to[s]$ by the location of the first unoccupied spot. There are $s^{n}$ functions $[n]\to[s],$ which we view as preference lists. If $\pi\in[s]^n$ is not a parking function, then there must be an unoccupied spot somewhere among the $n$ spots; all such unoccupied spots lie in $[s],$ since the last $n-s$ spots are always filled by the minimum-defect argument above.

	Let $i$ be the number of spots before the first unoccupied spot. Then there must be $i$ cars that form a parking function on those first $i$ spots. Choose those $i$ cars in one of $\binom{n}{i}$ ways, and one of the $(i + 1)^{i - 1}$ parking functions of length $i$. None of the remaining $n - i$ cars can prefer any of the first $i + 1$ spots; otherwise at least $i+1$ cars have preferences in the first $i+1$ spots, so at least one of them will be forced to proceed to spot $i+1$. However, they all can prefer any of the remaining $s - i - 1$ spots in one of $(s - i - 1)^{n - i}$ ways. This makes for $\binom{n}{i} (i + 1)^{i - 1}(s - i - 1)^{n - i}$ preference lists in $[s]^{n}$ so that the first unoccupied spot is $i + 1$. Subtracting the sum over all possible $i$ from the number of total preference lists gives the desired equality.
\end{proof}

This expression is one half of a sum given by Abel's binomial theorem, as noted in \cite{cameron-johannsen-prellberg-schweitzer-2008} (we expand on Abel's binomial theorem in \cref{sec:abelsbin}). They suggest there should be a different count corresponding to the other half of the sum. We present a novel combinatorial proof of this count by an involution argument.

\begin{restatable}{theorem}{resPFcount2}
    \label{thm:resPFcount2}
    The number of $[s]$-restricted parking functions on $n$ cars may also be expressed as
    \[\#\mathrm{PF}_{n|[s]}=\sum_{i = s}^{n} \binom{n}{i} (i + 1)^{i - 1} (s - i - 1)^{n - i}.\]
\end{restatable}

\begin{proof}
    Our goal here will be to take all parking functions on $n$ cars and eliminate those with cars that prefer spots in $[n] \setminus [s]$. We do so by means of an alternating sum that counts every parking function positively and negatively an equal number of times, \textit{unless} its preferences lie entirely within $[s].$

	We consider 2-colored parking functions --- that is, parking functions with cars assigned one of two colors --- with a particular restriction upon them.  Suppose we choose a subset of size $i \ge s$ of our cars to be colored indigo. These cars are to form a parking function on the first $i$ parking spots (though the parking outcome may not result in them actually parking there). The remaining $n - i$ cars are colored red and can prefer any of the $i + 1 - s$ ``forbidden'' spots in  $[i + 1] \setminus [s]$.
    
%\begin{figure}
%\begin{center}
%\begin{tikzpicture}[scale=0.3]
%	\mzncar[(0,0)]{blue}{blue} \mzncar[(3,0)]{red}{red} \mzncar[(6,0)]{blue}{blue} \mzncar[(9,0)]{blue}{blue} \mzncar[(12,0)]{red}{red}
%\end{tikzpicture} $\quad \longleftrightarrow \quad$
%\begin{tikzpicture}[scale=0.3]
%	\mzncar[(18,0)]{blue}{blue} \mzncar[(21,0)]{red}{red} \mzncar[(24,0)]{blue}{blue} \mzncar[(27,0)]{blue}{blue} \mzncar[(30,0)]{blue}{blue}
%\end{tikzpicture}
%\end{center}
%\caption{For $n=5$, $s = 2$, two colorings of $\pi = (1,3,2,2,4)$ with $i = 3$ and $i = 4$ are connected through the sign-reversing involution by recoloring the car with the largest preference. Note that both colorings satisfy the annotated constraints of the proof below.}
%\end{figure}

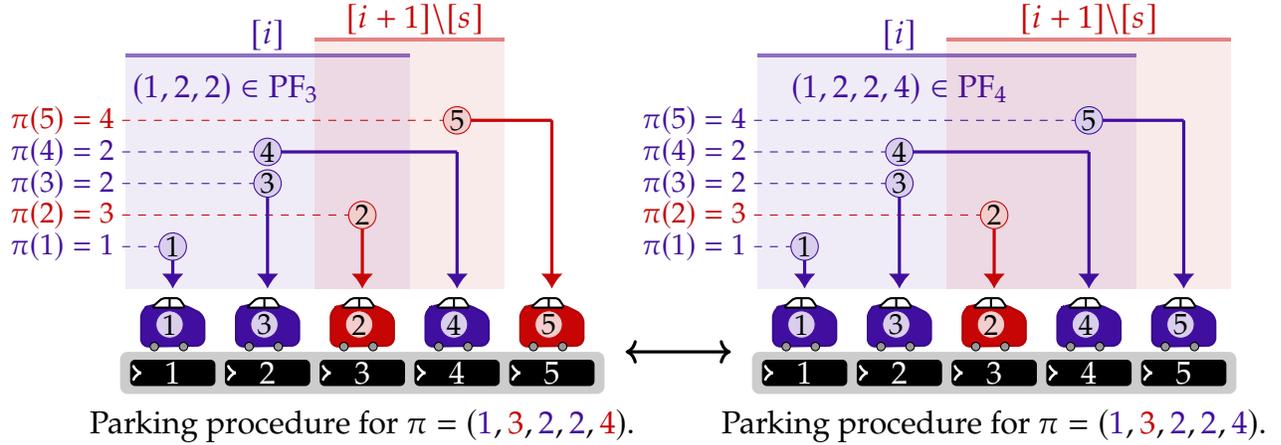
\begin{figure}[h]
\begin{center}
\begin{tikzpicture}[scale=0.28]
    \def\gap{5};
    \def\mid{2};
    \def\carsep{4.5};
    \def\vertsep{1.5};
    \definecolor{blue}{RGB}{65, 15, 160}
    \definecolor{red}{RGB}{200, 5, 5}
    
    \begin{scope}
    \filldraw[fill = blue, fill opacity = .08, draw=none] (-.25,3) rectangle (\carsep*3-.25, \gap+\vertsep*6);
    \draw[text = blue](\carsep, \gap+\vertsep*5) node {$(1,2,2) \in \mathrm{PF}_3$};
    \filldraw[fill = blue, fill opacity = .6, draw=none] (-.25,\gap+\vertsep*6 + .15) rectangle (\carsep*3-.25, \gap+\vertsep*6);
    \draw[text = blue](\carsep + \mid, \gap+\vertsep*6.7) node {$[i]$};

    \filldraw[fill = red, fill opacity = .08, draw=none] (\carsep*2-.25,3) rectangle (\carsep*4-.25, \gap+\vertsep*6.5);
    \filldraw[fill = red, fill opacity = .5, draw=none] (\carsep*2-.25,\gap+\vertsep*6.5 + .15) rectangle (\carsep*4-.25, \gap+\vertsep*6.5);
    \draw[text = red](\carsep*3, \gap+\vertsep*7.2) node {$[i+1]\backslash[s]$};
    % \draw[text = red](3*\carsep, \gap+\vertsep*7) node {$3,4 \in [4]\backslash[2]$};
    
    \draw[dashed,blue] (-.4, \gap + \vertsep*0) -- ++(\mid+\carsep*0,0); 
    \draw[dashed,red ] (-.4, \gap + \vertsep*1) -- ++(\mid+\carsep*2,0); 
    \draw[dashed,blue] (-.4, \gap + \vertsep*2) -- ++(\mid+\carsep*1,0); 
    \draw[dashed,blue] (-.4, \gap + \vertsep*3) -- ++(\mid+\carsep*1,0); 
    \draw[dashed,red ] (-.4, \gap + \vertsep*4) -- ++(\mid+\carsep*3,0); 

    \draw[-{Latex[length=2mm,width=3mm]}, very thick, blue] (\mid+\carsep*0, \gap+\vertsep*0) -- (\mid+\carsep*0 ,3);
    \draw[-{Latex[length=2mm,width=3mm]}, very thick, red]  (\mid+\carsep*2, \gap+\vertsep*1) -- (\mid+\carsep*2 ,3);
    \draw[-{Latex[length=2mm,width=3mm]}, very thick, blue] (\mid+\carsep*1, \gap+\vertsep*2) -- (\mid+\carsep*1 ,3);
    \draw[-{Latex[length=2mm,width=3mm]}, very thick, blue] (\mid+\carsep*3, \gap+\vertsep*3) -- (\mid+\carsep*3 ,3);
    \draw[-{Latex[length=2mm,width=3mm]}, very thick, red]  (\mid+\carsep*4, \gap+\vertsep*4) -- (\mid+\carsep*4 ,3);

    \draw[very thick, blue] (\mid+\carsep*1, \gap + \vertsep*3) -- ++(\carsep*2,0);
    \draw[very thick, red]  (\mid+\carsep*3, \gap + \vertsep*4) -- ++(\carsep*1,0);
    
    \tikpref[(2+\carsep*0, \gap + \vertsep*0)]{blue}{$1$} 
    \tikpref[(2+\carsep*2, \gap + \vertsep*1)]{red} {$2$} 
    \tikpref[(2+\carsep*1, \gap + \vertsep*2)]{blue}{$3$} 
    \tikpref[(2+\carsep*1, \gap + \vertsep*3)]{blue}{$4$} 
    \tikpref[(2+\carsep*3, \gap + \vertsep*4)]{red} {$5$}

    \mzncar[(\carsep*0,0)]{blue}{$1$} 
    \mzncar[(\carsep*1,0)]{blue}{$3$} 
    \mzncar[(\carsep*2,0)]{red}{$2$} 
    \mzncar[(\carsep*3,0)]{blue}{$4$} 
    \mzncar[(\carsep*4,0)]{red}{$5$}

    \draw[fill = black, fill opacity = .2, draw = none, rounded corners = .5 ex] (-.5,0) rectangle (\carsep*5, -2);
    
    \tikspot[(\carsep*0,0)]{$1$}
    \tikspot[(\carsep*1,0)]{$2$}
    \tikspot[(\carsep*2,0)]{$3$}
    \tikspot[(\carsep*3,0)]{$4$}
    \tikspot[(\carsep*4,0)]{$5$}

    \draw (-.25,\gap + \vertsep*0) node[left,blue]{\small$\pi(1) = 1$};
    \draw (-.25,\gap + \vertsep*1) node[left,red ]{\small$\pi(2) = 3$};
    \draw (-.25,\gap + \vertsep*2) node[left,blue]{\small$\pi(3) = 2$};
    \draw (-.25,\gap + \vertsep*3) node[left,blue]{\small$\pi(4) = 2$};
    \draw (-.25,\gap + \vertsep*4) node[left,red ]{\small$\pi(5) = 4$};

    % \draw (\carsep*1 + \mid, -2.25) node {$s$};
    % \draw (\carsep*2 + \mid, -2.25) node {$s+1$};
    % \draw (\carsep*2 + \mid, -3.25) node {$i$};
    % \draw (\carsep*3 + \mid, -3.25) node {$i+1$};

    \draw (\carsep*2 + \mid, -3.5) node {Parking procedure for $\pi = (\textcolor{blue}{1}, \textcolor{red}{3}, \textcolor{blue}{2}, \textcolor{blue}{2}, \textcolor{red}{4})$.};
    \end{scope}

    \draw[<->, very thick] (\carsep*5+1.5-.5, 0) -- (30 - 1.5, 0);

    \begin{scope}[shift = {(30,0)}]
    \filldraw[fill = blue, fill opacity = .08, draw=none] (-.25,3) rectangle (\carsep*4-.25, \gap+\vertsep*6);
    \draw[text = blue](\carsep + \mid, \gap+\vertsep*5) node {$(1,2,2,4) \in \mathrm{PF}_4$};
    \filldraw[fill = blue, fill opacity = .6, draw=none] (-.25,\gap+\vertsep*6 + .15) rectangle (\carsep*4-.25, \gap+\vertsep*6);
    \draw[text = blue](\carsep + \mid, \gap+\vertsep*6.7) node {$[i]$};

    \filldraw[fill = red, fill opacity = .08, draw=none] (\carsep*2-.25,3) rectangle (\carsep*5-.25, \gap+\vertsep*6.5);
    \filldraw[fill = red, fill opacity = .5, draw=none] (\carsep*2-.25,\gap+\vertsep*6.5 + .15) rectangle (\carsep*5-.25, \gap+\vertsep*6.5);
    \draw[text = red](\carsep*3+\mid, \gap+\vertsep*7.2) node {$[i+1]\backslash[s]$};
    % \draw[text = red](3*\carsep, \gap+\vertsep*7) node {$3,4 \in [4]\backslash[2]$};
    
    \draw[dashed,blue] (-0.4, \gap + \vertsep*0) -- ++(\mid+\carsep*0,0); 
    \draw[dashed,red ] (-0.4, \gap + \vertsep*1) -- ++(\mid+\carsep*2,0); 
    \draw[dashed,blue] (-0.4, \gap + \vertsep*2) -- ++(\mid+\carsep*1,0); 
    \draw[dashed,blue] (-0.4, \gap + \vertsep*3) -- ++(\mid+\carsep*1,0); 
    \draw[dashed,blue] (-0.4, \gap + \vertsep*4) -- ++(\mid+\carsep*3,0); 

    \draw[-{Latex[length=2mm,width=3mm]}, very thick, blue] (\mid+\carsep*0, \gap+\vertsep*0) -- (\mid+\carsep*0 ,3);
    \draw[-{Latex[length=2mm,width=3mm]}, very thick, red]  (\mid+\carsep*2, \gap+\vertsep*1) -- (\mid+\carsep*2 ,3);
    \draw[-{Latex[length=2mm,width=3mm]}, very thick, blue] (\mid+\carsep*1, \gap+\vertsep*2) -- (\mid+\carsep*1 ,3);
    \draw[-{Latex[length=2mm,width=3mm]}, very thick, blue] (\mid+\carsep*3, \gap+\vertsep*3) -- (\mid+\carsep*3 ,3);
    \draw[-{Latex[length=2mm,width=3mm]}, very thick, blue]  (\mid+\carsep*4, \gap+\vertsep*4) -- (\mid+\carsep*4 ,3);

    \draw[very thick, blue] (\mid+\carsep*1, \gap + \vertsep*3) -- ++(\carsep*2,0);
    \draw[very thick, blue]  (\mid+\carsep*3, \gap + \vertsep*4) -- ++(\carsep*1,0);
    
    \tikpref[(2+\carsep*0, \gap + \vertsep*0)]{blue}{$1$} 
    \tikpref[(2+\carsep*2, \gap + \vertsep*1)]{red} {$2$} 
    \tikpref[(2+\carsep*1, \gap + \vertsep*2)]{blue}{$3$} 
    \tikpref[(2+\carsep*1, \gap + \vertsep*3)]{blue}{$4$} 
    \tikpref[(2+\carsep*3, \gap + \vertsep*4)]{blue} {$5$}

    \mzncar[(\carsep*0,0)]{blue}{$1$} 
    \mzncar[(\carsep*1,0)]{blue}{$3$} 
    \mzncar[(\carsep*2,0)]{red}{$2$} 
    \mzncar[(\carsep*3,0)]{blue}{$4$} 
    \mzncar[(\carsep*4,0)]{blue}{$5$}

    \draw[fill = black, fill opacity = .2, draw = none, rounded corners = .5 ex] (-.5,0) rectangle (\carsep*5, -2);
    
    \tikspot[(\carsep*0,0)]{$1$}
    \tikspot[(\carsep*1,0)]{$2$}
    \tikspot[(\carsep*2,0)]{$3$}
    \tikspot[(\carsep*3,0)]{$4$}
    \tikspot[(\carsep*4,0)]{$5$}

    \draw (-.25,\gap + \vertsep*0) node[left,blue]{\small$\pi(1) = 1$};
    \draw (-.25,\gap + \vertsep*1) node[left,red ]{\small$\pi(2) = 3$};
    \draw (-.25,\gap + \vertsep*2) node[left,blue]{\small$\pi(3) = 2$};
    \draw (-.25,\gap + \vertsep*3) node[left,blue]{\small$\pi(4) = 2$};
    \draw (-.25,\gap + \vertsep*4) node[left,blue]{\small$\pi(5) = 4$};

    % \draw (\carsep*1 + \mid, -2.25) node {$s$};
    % \draw (\carsep*2 + \mid, -2.25) node {$s+1$};
    % \draw (\carsep*3 + \mid, -3.25) node {$i$};
    % \draw (\carsep*4 + \mid, -3.25) node {$i+1$};

    \draw (\carsep*2 + \mid, -3.5) node {Parking procedure for $\pi = (\textcolor{blue}{1}, \textcolor{red}{3}, \textcolor{blue}{2}, \textcolor{blue}{2}, \textcolor{blue}{4})$.};
    \end{scope}
\end{tikzpicture}
\caption{For $n = 5, s = 2$, two colorings of $(1,3,2,2,4)$ with $i = 3$ and $i = 4$ are connected through the sign reversing involution by recoloring the car with the largest preference. Note that both colorings satisfy the annotated constraints of the proof below.}
\end{center}
\label{fig:recoloring}
\end{figure}

    These preferences always form a parking function --- the Catalan condition is satisfied up to spot $i$ by the indigo cars since they form a parking function, and then the red cars all prefer one of the first $i + 1$ spots, satisfying the condition from $i+1$ onward.
    
    We count these parking functions positively when there are an even number of red cars and negatively otherwise. This count is
	\[
		\sum_{i = s}^{n} \binom{n}{i} (i + 1)^{i - 1} (i + 1 - s)^{n - i} (-1)^{n - i}.
	\]
	This is equal to the sum above; note that consecutive terms have opposite sign.

	We now create a sign-reversing involution on these 2-colored parking functions (Figure 1). If $m$ is the maximum preference among all the cars, let car $x$ be the first car that prefers spot $m$ (this choice is arbitrary, but one such car must be specified). That is, for a parking function $\pi$, let \[x = \min(\pi^{-1}(m)),\] where $m=\max(\pi([n]))$. Note that $m > s$ as long as there exist any red cars, since red cars all prefer spots greater than $s$. For our involution, we recolor this car, making it red if it was indigo and vice versa. 

	If car $x$ was red, then  $m\in [i + 1]$ by definition. Consider the same parking function with car $x$ painted indigo. Adding any preference in $[i + 1]$ to a parking function of length $i$ will form a parking function of length $i + 1$. To see this, consider adding the preference at the end. The first $i$ cars will park in the first $i$ spots (on which they form a parking function), and then the added car will park in the last spot. Since parking functions are permutation invariant, the list will form a parking function no matter where we add the car. Thus, with car $x$, the indigo cars form a parking function of length $i + 1$.
    
    Since the red cars with car $x$ formed a subset of $[i + 1] \setminus [s]$, without car $x$ they certainly form a subset of the larger $[i + 2] \setminus [s]$. Both conditions are thus satisfied for one direction of the involution.

	If car $x$ was indigo, then $m \in [i]$ (since the $i$ indigo cars form a parking function). Consider the same parking function with car $x$ painted red. The remaining indigo cars will still form a parking function on $i - 1$ cars: the $j$th smallest preference will still be at most $j$ after removing the maximum preference, so the Catalan condition is still satisfied.
    
    Since $m \in [i]$ was the maximum preference, all red cars' preferences will be in $[i].$ No old red cars will be in $[s],$ since all were contained in $[i+1]\backslash [s].$ The new red car is the only thing to worry about: all red cars will have preferences in $[i]\backslash [s],$ and therefore all conditions will be satisfied for the involution, unless $m\in [s].$

	This is our exception --- we cannot recolor to have a new red car if even the maximum preference is in $[s]$. This corresponds to the $[s]$-restricted parking functions of length $n$. They are only counted once --- positively, as a subset of the term with 0 red cars. Any other pair of 2-colored parking functions is counted positively once and negatively once in our sum, for an even and odd number of red cars respectively. Thus, the sum counts only $\#\mathrm{PF}_{n \mid [s]}$.
\end{proof}

\subsection{\texorpdfstring{$[s]$-r}{R}estricted prime parking functions}

There is an analogous result for prime parking functions as well. Recall that prime parking functions $\pi:[n]\to[n]$ satisfy an extra-strong Catalan condition --- that more than $i$ cars prefer one of the first in $i$ spots for each $i \le n$. We will use the fact that their total count is $(n-1)^{n-1}$. Again, we refer the reader to \cite{armon-2024} for a more detailed discussion.

We can enumerate $[s]$-restricted prime parking functions (defined in \cref{primesection} as those prime parking functions that are also $[s]$-restricted) in much the same manner as we did normal restricted parking functions:

\begin{theorem}
    \label{thm:resPPFcount1}
    For $1 \le s < n,$ the number of $[s]$-restricted prime parking functions on $n$ cars is 
    \[\#\mathrm{PPF}_{n|[s]}=s^{n} - (s - 1)^{n} - \sum_{i = 1}^{s} \binom{n}{i} (i - 1)^{i - 1} (s - i)^{n - i}.\]
\end{theorem}

\begin{proof}
	Our goal is again to take all $s^n$ functions with restricted codomain and remove those which are not prime parking functions. A prime parking function must have its 1st, 2nd, 3rd, 4th, etc. smallest preferences be at most $1,1,2,3,\ldots$; we do casework on the first counterexample. We treat the first case separately: there are $(s-1)^n$ ways to have one's smallest preference greater than 1, as we must choose preferences from $2$ to $s$ for each car.
    
    Now we can assume that some $i\in [s-1]$ is the smallest value such that the $i+1$th value is greater than $i$. Since more than $j$ cars prefer the first $j$ spots for all $j < i+1$, the $i$ cars that prefer the first $i$ spots form a prime parking function on $[i]$ in one of $(i - 1)^{i - 1}$ ways. The remaining $n - i$ cars can prefer any of the remaining $s - i$ spots in $[s] \setminus [i]$ in one of $(s - i)^{n - i}$ ways. With $\binom{n}{i}$ ways to choose the $i$ cars, there are $\binom{n}{i} (i - 1)^{i - 1} (s - i)^{n - i}$ preference lists in $[s]^{n}$ so that the first counterexample is $i+1$.

    This means that subtracting off all possible non-prime parking functions gives the desired equality.
\end{proof}

Also similarly to classical parking functions, we can enumerate $[s]$-restricted prime parking functions by involution. This appears to give a different result, but we will see in \cref{sec:abelsbin} that it is equivalent by Abel's binomial theorem.

\begin{theorem}
    \label{thm:resPPFcount2}
    For $1 \le s < n,$ the number of $[s]$-restricted prime parking functions on $n$ cars may also be expressed as
    \[\#\mathrm{PPF}_{n|[s]}=\sum_{i = s + 1}^{n} \binom{n}{i} (i - 1)^{i - 1} (s - i)^{n - i}.\]
\end{theorem}

\begin{proof}
    We proceed similarly to \cref{thm:resPFcount2}, with modifications to account for prime parking functions.
    
    We will now count 2-colored prime parking functions. The $i$ indigo cars are to form a prime parking function on $[i]$; the remaining $n - i$ cars are colored red and can prefer any of the $i - s$ forbidden spots in  $[i] \setminus [s]$. These preferences always form a prime parking function --- the extra-strong Catalan condition is satisfied up to spot $i-1$ by the indigo cars since they form a prime parking function, and then the red cars all prefer one of the first $i$ spots, satisfying the condition from $i$ onwards. We count these parking functions positively with an even number of red cars and negatively otherwise. This count is
	\[
		\sum_{i = s}^{n} \binom{n}{i} (i - 1)^{i - 1} (i - s)^{n - i} (-1)^{n - i};
	\]
	this is equal to the sum above, noting that the $i=s$ term becomes 0.

	We use the same involution as before: recolor car $x,$ the first car that has the maximum preference $m.$ Note that $m > s$ as long as there are some red cars, since the red cars all prefer spots greater than $s$.

	If car $x$ is painted red, then  $m \in [i]$ just by definition. We can reconstruct the same parking function with car $x$ painted indigo. We can add any preference in $[i]$ to a prime parking function of length $i$ and the list will form a prime parking function of length $i + 1,$ by the same argument as above.
    
    Since the red cars with car $x$ formed a subset of $[i] \setminus [s]$, without car $x$ they also form a subset of the larger $[i + 1] \setminus [s]$. Thus all conditions are satisfied and our involution works in one direction.

	If car $x$ is painted indigo, then $m\in [i-1]$ for the indigo cars to form a prime parking function on $i$ cars. Thus, we can reconstruct the same parking function with car $x$ repainted red. The remaining indigo cars will still form a prime parking function on $i - 1$ cars --- the $j$th smallest preference will still be less than $j$ after removing the maximum preference.
    
    Since $m\in [i-1]$ was the maximum preference, all red cars will be in $[i-1]$. No old red cars will be in $[s],$ since all were contained in $[i]\backslash [s].$ The new red car is the only thing to worry about: all red cars will have preferences in $[i-1]\backslash [s],$ and therefore all conditions will be satisfied for the involution, unless $m\in [s].$

	Once again, then, we have a single exception to the involution: prime parking functions where all preferences are in $[s]$ are counted once while everything else cancels. Thus our alternating sum counts $\mathrm{PPF}_{n \mid [s]}$.
\end{proof}

\subsection{Abel's binomial theorem} \label{sec:abelsbin}

In enumerating $[s]$-restricted parking functions and their prime variants, we demonstrated two different counts for each. One was obtained by excluding all of the restricted preference lists that aren't parking functions and another by canceling all parking functions with forbidden preferences. The expressions we get from these two approaches are very similar. This isn't a coincidence! These pairs of counts are actually halves of Abel's binomial theorem.

Abel's binomial theorem is a generalization of the standard binomial theorem, with deep roots in the theory of reluctant functions and forests (see \cite{shapiro-1991} and \cite{pitman-2002} for more on these connections). The identity has many equivalent statements. One common form similar to those in \cite{shapiro-1991} and \cite{zucker-2024} is
\[
	(X + Y)^{n} = \sum_{k = 0}^{n} \binom{n}{k} X (X + kZ)^{k - 1} (Y - k Z)^{n - k}
\]
for $X, Y, Z \in \mathbb{C}$ and $n \in \mathbb{Z}_{> 0}$. The substitutions $x = X/Z$ and $y = Y/Z - n$ (for suitable $Z$) yield an equivalent form that is particularly amenable to combinatorial interpretation.

\begin{restatable}[Abel's binomial theorem]{theorem}{abelsbin}
	For all $n\in\mathbb{Z}_{>0}$ and $x,y\in\mathbb{C},$
    \[
		(x + y + n)^{n} = \sum_{i = 0}^{n} \binom{n}{i} x (x + i)^{i - 1} (y + n - i)^{n - i}.
	\]
\end{restatable}

Since the expressions in \cref{thm:resPFcount1} and \cref{thm:resPFcount2} both count $[s]$-restricted parking functions on $n$ cars, we can set them equal. Rearranging and combining sums gives
\begin{align*}
    s^{n} &= \sum_{i = 0}^{s - 1} \binom{n}{i} (i + 1)^{i - 1} (s - i - 1)^{n - i} + \sum_{i = s}^{n} \binom{n}{i} (i + 1)^{i - 1} (s - i - 1)^{n - i}\\
    &=\sum_{i = 0}^{n} \binom{n}{i} (i + 1)^{i - 1} (s - i - 1)^{n - i},
\end{align*}
which is Abel's theorem for $x=1,$ $y=s-n-1.$ Similarly, from \cref{thm:resPPFcount1} and \cref{thm:resPPFcount2} we get 
\begin{align*}
    - (s - 1)^{n} &= -s^n + \sum_{i = 1}^{s} \binom{n}{i} (i - 1)^{i - 1} (s - i)^{n - i} + \sum_{i = s + 1}^{n} \binom{n}{i} (i - 1)^{i - 1} (s - i)^{n - i}\\
    &=\sum_{i=0}^n \binom{n}{i} (i - 1)^{i - 1} (s - i)^{n - i},
\end{align*}
which is Abel's theorem for $x=-1$ and $y=s-n+1$. This relationship is very surprising, especially given one half of the sum is alternating and the other isn't!

These are merely special cases, but we can actually recover the complete form of the identity from parking functions. To do so, we'll need a result found in \cite{yan-survey-2015}, about the number-of-ones enumerator.

\begin{definition}
	The \emph{number-of-ones enumerator} for parking functions on $n$ cars is the generating function for the number of parking functions with $i$ cars preferring the first spot.
	\[
		\sum_{\pi \in \mathrm{PF}_{n}} x^{\# \pi^{-1}(\{ 1 \})} = \sum_{i = 1}^{n} c_{i} x^{i}
	\]
	where $c_{i} = \# \{ \pi \in \mathrm{PF}_{n} \mid \# \pi^{-1}(\{ 1 \}) = i \}$.
\end{definition}

\begin{lemma}
    For each $n\in\mathbb{Z}_{>0}$, the number-of-ones enumerator for parking functions of length $n$ is $x(x + n)^{n - 1}$.
\end{lemma}

Yan demonstrates a clever construction involving a bijection with labeled trees to prove this. We can use this result to prove something analogous for the corresponding enumerator for $[s]$-restricted parking functions. 

\begin{definition}
	The \emph{$[s]$-restricted number-of-ones enumerator} for parking functions on $n$ cars is the generating function for the number of parking functions with $i$ cars preferring the first spot.
	\[
		\sum_{\pi \in \mathrm{PF}_{n \mid [s]}} x^{\# \pi^{-1}(\{ 1 \})} = \sum_{i = 1}^{n} c_{i} x^{i}
	\]
	where $c_{i} = \# \{ \pi \in \mathrm{PF}_{n \mid [s]} \mid \# \pi^{-1}(\{ 1 \}) = i \}$.
\end{definition}

We can obtain a suspiciously Abel-ian pair of expressions for the $[s]$-restricted number-of-ones enumerator.

\begin{theorem}
    \label{thm:res-1s-enumerator}
    The $[s]$-restricted number-of-ones enumerator for parking functions on $n$-cars may be expressed in two forms:
    \begin{align*}
    \sum_{\pi\in\mathrm{PF}_{n|[s]}}x^{\# \pi^{-1}(\{1\})}&=(s-1+x)^{n} - \sum_{i = 0}^{s - 1} \binom{n}{i} x(x + i)^{i - 1} (s - i - 1)^{n - i}\\
    &=\sum_{i = s}^{n} \binom{n}{i} x(x + i)^{i - 1} (s - i - 1)^{n - i}.
    \end{align*}
\end{theorem}

\begin{proof}
    We omit most of the details here, as they follow along lines very similar to \cref{thm:resPFcount1} and \cref{thm:resPFcount2}.
    
    For the first claim, $(s-1+x)^n$ is a generating function for all functions restricted to $[s],$ as the number of functions with $k$ 1s in its output is $\binom{n}{k}1^k(s-1)^{n-k}.$ In each term of the sum we're subtracting off, $x(x+i)^{i-1}$  counts parking functions of length $i$; the $(s-i-1)^{n-i}$ represents preferences after the first gap in the street, and therefore it has no 1s to account for.

    The second claim proceeds similarly; $x(x+i)^{i-1}$ again counts parking functions of length $i$. Since $(s-i-1)^{n-i}$ counts red cars, which always prefer forbidden spots greater than $s\ge1,$ we again needn't account for preferences of 1 here. From this point on the proofs are identical.
\end{proof}

The pair of expressions above are actually sufficient to recover all of Abel's binomial theorem.

\begin{proof}[Proof of Abel's binomial theorem]
	Equating the two expressions in the previous theorem gives us that for all $n\in\mathbb{Z}_{>0},$ $s\in [n],$ $x\in\mathbb{C},$
	\[
		(s-1+x)^{n} = \sum_{i = 0}^{n} \binom{n}{i} x(x + i)^{i - 1} (s - i - 1)^{n - i}.
	\]
In particular, viewing both sides of the expression as a polynomial in $s$: for any $n\in\mathbb{Z}_{>0}$ and $x\in\mathbb{C},$ we have two monic polynomials $f(s)$ and $g(s)$ of degree $n$ which agree on $n$ distinct values. They must be equal; $f(s)-g(s)$ is of degree at most $n-1$ and has at least $n$ roots and so must be the zero polynomial.

Thus, the above identity is true for all $x,s\in\mathbb{C}$ and $n\in\mathbb{Z}_{>0}$. Substituting $s=n+1+y$ recovers Abel's binomial theorem.
\end{proof}

Note that this result essentially comes from looking at the ``number-of-ones" statistic on $[s]$-restricted parking functions. Other statistics, such as lucky and area, could be promising as well. See \cite{yan-survey-2015} for an overview.

\subsection{The symmetric group and initial restrictions}

Regular parking functions (on $n$ cars) have certain structure beyond just whether they are parking functions or not. For example, whether a given $\pi : [n] \to [n]$ is a parking function depends only on the size of each of the pre-images $\pi^{-1}([i])$. It doesn't matter which cars prefer spots in each $[i]$. However, the parking function itself also comes with the data of its \emph{parking outcome}, the permutation $[n] \to [n]$ sending each parking spot to the car parked in it, which clearly does depend on which cars prefer which spots. This structure involves the symmetric group $\mathfrak{S}_{n}$ in two different ways.

First, there is an action of the symmetric group on $\mathrm{PF}_{n}$ by shuffling around the preferred spots among the cars. Explicitly, this is the group action $\mathfrak{S}_{n} \times \mathrm{PF}_{n} \to \mathrm{PF}_{n}$ by $\sigma \cdot \pi = \pi \circ \sigma^{-1}$. Counting the orbits of this action is equivalent to counting the ``non-decreasing'' (order-preserving) parking functions. It's known that they are in bijection with Dyck paths and thus, the number of them is the $n$th Catalan number $C_{n}$ (see, for example, \cite{armstrong-loehr-warrington-2016}).

Second, there is a surjection $\mathrm{PF}_{n} \twoheadrightarrow \mathfrak{S}_{n}$ that sends every parking function to its parking outcome. It is interesting to consider the distribution of parking outcomes. In \cite{pinsky-2024}, Pinsky determines the size of the pre-image of a permutation $\sigma \in \mathfrak{S}_{n}$ under the map $\mathrm{PF}_{n} \twoheadrightarrow \mathfrak{S}_{n}$, and analyzes asymptotics as $n \to \infty$.

Both of these facts have analogues for $[s]$-restricted parking functions. 

\begin{theorem}
	Under the described action of $\mathfrak{S}_{n}$ on $\mathrm{PF}_{n \mid [s]}$, the number of orbits is $C(n, s - 1)$ --- the $(n, s - 1)$th entry in Catalan's triangle. \cite[A009766]{oeis}
\end{theorem} 

Note, as a corollary, the number of non-decreasing parking functions such that $s$ is the maximal spot preferred is $C(n, s - 1) - C(n, s - 2) = C(n - 1, s - 1)$. In symbols, $\# \{ \pi^{\uparrow} \in \mathrm{PF}^{\uparrow}_{n} \mid \max \pi^{\uparrow}(j) = s, j \in [n] \} = C(n - 1, s - 1)$. This result is stated or proved in other sources. Cai and Yan state the result without proof \cite{cai-yan-2019}. Cai and Yan also suggest bijections with different classes of Dyck paths and lattice paths, classes of forests, and other classes of non-decreasing parking functions. In the preprint \cite{cruz-harris-et-al-2024} the no-peak no-tie parking functions are shown to be in bijection with these parking functions and enumerated. The proof of below is independent of the proof in \cite{cruz-harris-et-al-2024}, and like all things we have shown so far with preference-restricted parking functions, provides novel combinatorial insight.

\begin{proof}
	It will suffice to count one specific representative of each orbit. We can count those representatives such that $\pi$ is order-preserving, called the \emph{non-decreasing $[s]$-restricted parking functions} because they give $\pi(1) \le \dots \le \pi(n)$. These are exactly those $[s]$-restricted parking functions that are non-decreasing (as defined in \cref{sec:intro}). Let $\operatorname{PF}_{n \mid [s]}^{\uparrow}$ be the subset of non-decreasing $[s]$-restricted parking functions.

	Consider any non-decreasing parking function $\pi \in \mathrm{PF}_{n \mid [s]}^{\uparrow}$. If the last car prefers the $s$th spot (if $\pi(n) = s$), then the remaining cars must form a non-decreasing parking function on $n - 1$ cars since they satisfy the required Catalan inequality $\pi(i) \le \min \{ i, s \}$. $\pi$ is determined by this choice in $\mathrm{PF}_{n - 1 \mid [s]}^{\uparrow}$ If the last car does not prefer the $s$th spot (if $\pi(n) < s$), then $\pi$ is just a parking function on $n$ cars restricted to $[s - 1]$ since $\pi(1) \le \dots \pi(n) \le s - 1 < s$. Again, $\pi$ is determined by the choice in $\mathrm{PF}_{n \mid [s - 1]}^{\uparrow}$.

	Any $[s]$-restricted parking function on $n - 1$ cars can be extended to an $[s]$-restricted parking function on $n$ cars by defining $\pi(n) = s$, and any $[s - 1]$-restricted parking function on $n$ cars is obviously also an $[s]$-restricted parking function. Since the two possibilities are mutually exclusive, we have the recurrence relation
	\[
		\# \operatorname{PF}_{n \mid [s]}^{\uparrow} = \# \operatorname{PF}_{n - 1 \mid [s]}^{\uparrow} + \# \operatorname{PF}_{n \mid [s - 1]}^{\uparrow} .
	\]
	Catalan's triangle satisfies the same recurrence relation, as shown in \cite{bailey-1996}:
	\[
		C(n, s - 1) = C(n - 1, s - 1) + C(n, s - 2).
	\]
	It only remains to show that the two sequences satisfy the same initial conditions.

	If the restriction is to $S = [1]$ there is only $1$ parking function $(1, \dots, 1) \in \operatorname{PF}_{n \mid [1]}^{\uparrow}$. Thus $\# \operatorname{PF}_{n \mid [1]}^{\uparrow} = 1 = C(n, 1 - 1)$ for all $n \in \mathbb{Z}_{> 0}$. Finally, we already know $\# \operatorname{PF}_{n \mid [n]}^{\uparrow}  = C_{n}$, the $n$th Catalan number. The Catalan triangle has $C(n, n - 1) = C_{n}$, and thus, we have $\# \operatorname{PF}_{n \mid [n]}^{\uparrow} = C(n, n - 1)$. These initial conditions are also found in \cite{bailey-1996}.

	With the same initial conditions on the diagonal and vertical and the same recurrence relation (giving a number in terms of those above and to its left) for both of the sequences, they must be the same. Thus, the number of $\mathfrak S_{n}$ orbits is $\# \operatorname{PF}_{n \mid [s]}^{\uparrow} = C(n, s - 1)$.
\end{proof}

\begin{theorem}[inspired by \cite{pinsky-2024}]
    The number of $[s]$-restricted parking functions on $n$ cars with parking outcome $\sigma \in \mathfrak{S}_n$ is
    \[
    \prod_{i = 1}^n \max \{ 0, \ell_{n,i}(\sigma)-\max \{ 0, i - s \} \}.
    \]
    Here $\ell_{n, i}$ is the length of the longest contiguous string in $\sigma$ (written in one-line notation as $\sigma_{1} \dots \sigma_{n}$) ending at $\sigma_i$ with maximum value $\sigma_i$ in the string.
\end{theorem}

\begin{proof}
	Recall that the parking outcome $\sigma$ sends each parking spot to the car that parks in it. Thus, writing $\sigma$ in one-line notation $\sigma_{1} \dots \sigma_{n}$ is just writing the car numbers in order of which parking spot they end up in.

	First consider the case of parking outcomes $\sigma$ of regular parking functions. In order for car $\sigma_{i}$ to end up in parking spot $i$, it is necessary and sufficient that it prefers a spot $j \le i$, such that all the spots between $j$ and $i$ are already occupied by the time it tries to park. Spots $j, j + 1, \dots, i$ need to be occupied by cars with lower numbers than $\sigma_{i}$ --- the string $\sigma_{j} \dots \sigma_{i}$ must have maximum $\sigma_{i}$. The longest such string contains all the parking spots that a car $\sigma_{i}$ could prefer to cause the given parking outcome. Thus, the number of parking functions with the given parking outcome is just the product of all $\ell_{n, i}$.

	However, in the case of $[s]$-restricted parking functions, when car $\sigma_{i}$ ends up in spot $i$ outside the first $s$ spots, it couldn't have preferred any of the $i - s$ spots $s + 1, s + 2, \dots, i$. Thus, to obtain the number of possible preferred spots, we subtract $\max \{ 0, i - s \}$ from the length of the longest string ending in in $\sigma_{i}$ with maximum value $\sigma_{i}$. If $\ell_{m, i}(\sigma) - \max \{ 0, i - s \}$ is negative, then $\sigma_{s}, \sigma_{s + 1}, \dots, \sigma_{n}$ is not in increasing order. This can only happen if some car $\sigma_{j}$ has a preferred spot in the last $n - s$ spots. Any such parking outcome $\sigma$ does not result from an $[s]$-restricted parking function. Thus, the number of possible preferred spots of $\sigma_{i}$ is actually $\max \{ 0, \ell_{n, i}(\sigma) - \max \{ 0, i - s \} \}$, and the number of $[s]$-restricted parking functions inducing parking outcome $\sigma$ is the product of all these terms.
\end{proof}
 
\section{Modular restrictions \label{modularsection}}

In \cite{blake-konheim-1977}, Blake and Konheim consider a variant of the parking problem applicable to hash buckets. In the language of cars and parking spots (they think about balls and buckets) they consider the problem of parking cars in a street, where instead of a parking spots each fitting a single car, there are rows which can each fit $g\ge 1$ cars, except for the last row, which potentially fits fewer. Here, cars have preferences for rows; they attempt to park in the first empty spot of their preferred row, and then each subsequent row until they find one that is not full. Using machinery from complex analysis and generating functions, Blake and Konheim obtain an enumeration of the ``parking functions'' --- functions describing preferences that allow every car to park --- under this modified parking procedure (\cite{blake-konheim-1977} Corollary 2.1). The same setup is described engagingly in \cite{carlson-harris-2020}.

We obtain the same enumeration by elementary means that provide an explicit generalization to the case of $gs - k$ spots as well. Instead of considering $gs - 1$ parking spots organized in $s$ rows of $g$ parking spots each (with $g - 1$ in the last row), we imagine placing these rows one after each other to form a long one-way street and only permit cars to prefer the first spot in each row. This parking procedure is exactly equivalent but can be understood as a restricted parking function. By adding $k$ spots to make our street a circular one with $gs$ spots, we can use the circular symmetry of the new setup to enumerate parking functions, in the same vein as Pollak's enumeration of parking functions (recounted by Riordan in \cite{riordan-1969}).

\begin{figure}[h]
    \centering
    \begin{tikzpicture}[scale=0.28]
    \def\gap{5};
    \def\mid{2};
    \def\carsep{4.5};
    \def\vertsep{2};
    \def\rowsep{1.5};
    \definecolor{amber}{rgb}{1.0, 0.75, 0.0}
    \definecolor{darkblue}{RGB}{37, 51, 186}
    \definecolor{darkgreen}{RGB}{37, 186, 98}
    \def\rowac{amber!70!black};
    \def\rowbc{darkgreen};
    \def\rowcc{darkblue};

    % Rows
    \draw[fill=\rowac,fill opacity=.2,draw=none] (-.5,-2) rectangle (\carsep*3, \gap+\vertsep*7);
    \draw[fill=\rowac,fill opacity=.8,draw=none] (-.5,\gap+\vertsep*7+.25) rectangle (\carsep*3, \gap+\vertsep*7);
    \draw (\carsep + \mid, \gap+\vertsep*7+.25) node[above, text = \rowac]{\textbf{Row 1}};
    \draw[fill=\rowbc,fill opacity=.2,draw=none] (\carsep*3+\rowsep-.5,-2) rectangle (\carsep*6 + \rowsep, \gap+\vertsep*7);
    \draw[fill=\rowbc,fill opacity=.8,draw=none] (\carsep*3+\rowsep-.5,\gap+\vertsep*7+.25) rectangle (\carsep*6 + \rowsep, \gap+\vertsep*7);
    \draw (\carsep*4 + \rowsep + \mid, \gap+\vertsep*7+.25) node[above, text = \rowbc]{\textbf{Row 2}};
    \draw[fill=\rowcc,fill opacity=.2,draw=none] (\carsep*6+\rowsep*2-.5,-2) rectangle (\carsep*7+\rowsep*2, \gap+\vertsep*7);
    \draw[fill=\rowcc,fill opacity=.8,draw=none] (\carsep*6+\rowsep*2-.5,\gap+\vertsep*7+.25) rectangle (\carsep*7+\rowsep*2, \gap+\vertsep*7);
    \draw (\carsep*6 + 2*\rowsep + \mid, \gap+\vertsep*7+.25) node[above, text = \rowcc]{\textbf{Row 3}};
    
    % Downward arrows
    \draw[-{Latex[length=2mm,width=3mm]}, very thick, \rowac] (\mid+\carsep*0+\rowsep*0, \gap+\vertsep*0) -- (\mid+\carsep*0+\rowsep*0,3);
    \draw[-{Latex[length=2mm,width=3mm]}, very thick, \rowac] (\mid+\carsep*1+\rowsep*0, \gap+\vertsep*3) -- (\mid+\carsep*1+\rowsep*0,3);
    \draw[-{Latex[length=2mm,width=3mm]}, very thick, \rowac] (\mid+\carsep*2+\rowsep*0, \gap+\vertsep*4) -- (\mid+\carsep*2+\rowsep*0,3);
    \draw[-{Latex[length=2mm,width=3mm]}, very thick, \rowbc] (\mid+\carsep*3+\rowsep*1, \gap+\vertsep*1) -- (\mid+\carsep*3+\rowsep*1,3);
    \draw[-{Latex[length=2mm,width=3mm]}, very thick, \rowbc] (\mid+\carsep*4+\rowsep*1, \gap+\vertsep*2) -- (\mid+\carsep*4+\rowsep*1,3);
    \draw[-{Latex[length=2mm,width=3mm]}, very thick, \rowac] (\mid+\carsep*5+\rowsep*1, \gap+\vertsep*5) -- (\mid+\carsep*5+\rowsep*1,3);
    \draw[-{Latex[length=2mm,width=3mm]}, very thick, \rowcc] (\mid+\carsep*6+\rowsep*2, \gap+\vertsep*6) -- (\mid+\carsep*6+\rowsep*2,3);

    % horizontal lines
    \draw[very thick, \rowac] (2+\carsep*0+\rowsep*0, \gap + \vertsep*0) -- ++(\carsep*0,0);
    \draw[very thick, \rowbc] (2+\carsep*3+\rowsep*1, \gap + \vertsep*1) -- ++(\carsep*0,0);
    \draw[very thick, \rowbc] (2+\carsep*3+\rowsep*1, \gap + \vertsep*2) -- ++(\carsep*1,0);
    \draw[very thick, \rowac] (2+\carsep*0+\rowsep*0, \gap + \vertsep*3) -- ++(\carsep*1,0);
    \draw[very thick, \rowac] (2+\carsep*0+\rowsep*0, \gap + \vertsep*4) -- ++(\carsep*2,0);
    \draw[very thick, \rowac] (2+\carsep*0+\rowsep*0, \gap + \vertsep*5) -- ++(\carsep*5+\rowsep,0);
    \draw[very thick, \rowcc] (2+\carsep*6+\rowsep*2, \gap + \vertsep*6) -- ++(\carsep*0,0);

    % Dashed lines
    \draw[dashed] (-.7, \gap + \vertsep*0) -- ++(\mid+\carsep*0+\rowsep*0,0); 
    \draw[dashed] (-.7, \gap + \vertsep*1) -- ++(\mid+\carsep*3+\rowsep*1,0); 
    \draw[dashed] (-.7, \gap + \vertsep*2) -- ++(\mid+\carsep*3+\rowsep*1,0); 
    \draw[dashed] (-.7, \gap + \vertsep*3) -- ++(\mid+\carsep*0+\rowsep*0,0); 
    \draw[dashed] (-.7, \gap + \vertsep*4) -- ++(\mid+\carsep*0+\rowsep*0,0); 
    \draw[dashed] (-.7, \gap + \vertsep*5) -- ++(\mid+\carsep*0+\rowsep*0,0); 
    \draw[dashed] (-.7, \gap + \vertsep*6) -- ++(\mid+\carsep*6+\rowsep*2,0); 

    % car preference circles
    \tikpref[(2+\carsep*0+\rowsep*0, \gap + \vertsep*0)]{\rowac}{$1$} 
    \tikpref[(2+\carsep*3+\rowsep*1, \gap + \vertsep*1)]{\rowbc}{$2$} 
    \tikpref[(2+\carsep*3+\rowsep*1, \gap + \vertsep*2)]{\rowbc}{$3$} 
    \tikpref[(2+\carsep*0+\rowsep*0, \gap + \vertsep*3)]{\rowac}{$4$} 
    \tikpref[(2+\carsep*0+\rowsep*0, \gap + \vertsep*4)]{\rowac}{$5$}
    \tikpref[(2+\carsep*0+\rowsep*0, \gap + \vertsep*5)]{\rowac}{$6$}
    \tikpref[(2+\carsep*6+\rowsep*2, \gap + \vertsep*6)]{\rowcc}{$7$}

    % parked cars
    \mzncar[(\carsep*0+\rowsep*0,0)]{\rowac}{$1$} 
    \mzncar[(\carsep*1+\rowsep*0,0)]{\rowac}{$4$} 
    \mzncar[(\carsep*2+\rowsep*0,0)]{\rowac}{$5$} 
    \mzncar[(\carsep*3+\rowsep*1,0)]{\rowbc}{$2$} 
    \mzncar[(\carsep*4+\rowsep*1,0)]{\rowbc}{$3$}
    \mzncar[(\carsep*5+\rowsep*1,0)]{\rowac}{$6$}
    \mzncar[(\carsep*6+\rowsep*2,0)]{\rowcc}{$7$}

    % street
    \draw[fill=black,fill opacity=.2,draw = none,rounded corners=.5 ex] (-.5,0) rectangle (\carsep*7 + \rowsep*2, -2);
    
    % spots
    \tikspotc[(\carsep*0 + \rowsep*0,0)]{$1$}{\rowac!25!black}
    \tikspotc[(\carsep*1 + \rowsep*0,0)]{$2$}{\rowac!25!black}
    \tikspotc[(\carsep*2 + \rowsep*0,0)]{$3$}{\rowac!25!black}
    \tikspotc[(\carsep*3 + \rowsep*1,0)]{$4$}{\rowbc!25!black}
    \tikspotc[(\carsep*4 + \rowsep*1,0)]{$5$}{\rowbc!25!black}
    \tikspotc[(\carsep*5 + \rowsep*1,0)]{$6$}{\rowbc!25!black}
    \tikspotc[(\carsep*6 + \rowsep*2,0)]{$7$}{\rowcc!25!black}

    \draw (-.75,\gap + \vertsep*0) node[left]{$\pi(1)$ : \textcolor{black}{spot $1$}, \textcolor{\rowac}{row $1$}};
    \draw (-.75,\gap + \vertsep*1) node[left]{$\pi(2)$ : \textcolor{black}{spot $4$}, \textcolor{\rowbc}{row $2$}};
    \draw (-.75,\gap + \vertsep*2) node[left]{$\pi(3)$ : \textcolor{black}{spot $4$}, \textcolor{\rowbc}{row $2$}};
    \draw (-.75,\gap + \vertsep*3) node[left]{$\pi(4)$ : \textcolor{black}{spot $1$}, \textcolor{\rowac}{row $1$}};
    \draw (-.75,\gap + \vertsep*4) node[left]{$\pi(5)$ : \textcolor{black}{spot $1$}, \textcolor{\rowac}{row $1$}};
    \draw (-.75,\gap + \vertsep*5) node[left]{$\pi(6)$ : \textcolor{black}{spot $1$}, \textcolor{\rowac}{row $1$}};
    \draw (-.75,\gap + \vertsep*6) node[left]{$\pi(7)$ : \textcolor{black}{spot $7$}, \textcolor{\rowcc}{row $3$}};

    % \draw (\carsep*2 + \mid, -3.5) node {Parking procedure for $\pi = (\textcolor{blue}{1}, \textcolor{red}{3}, \textcolor{blue}{2}, \textcolor{blue}{2}, \textcolor{red}{4})$.};
\end{tikzpicture}
    \caption{A row parking function can be understood as a restricted parking function through their identical parking procedures. As a row parking function, cars prefer rows $(1,2,2,1,1,3,1)$, with row capacities $(3,3,1)$. As a restricted parking function, $\pi = (1, 4, 4, 1, 1, 7, 1) \in \mathrm{PF}_{7 \mid \{ 1, 4, 7 \}}$. This is an example of the case $g = s = 3$, $k = 2$.}
    \label{fig:row-parking}
\end{figure}
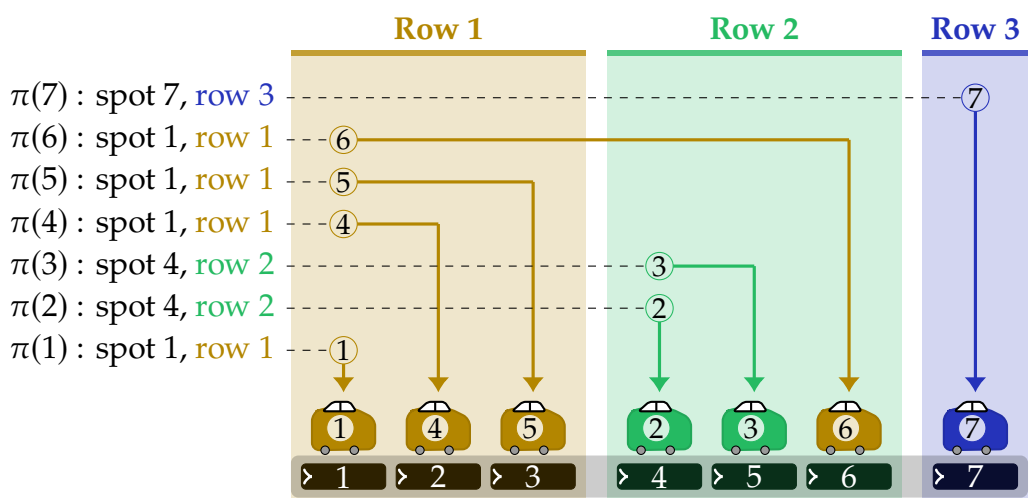

\begin{restatable}{theorem}{modPFcount}
	\label{thm:modPFcount}
	The number of parking functions of length $gs - 1$ with gap $g$ between possible preferred spots is
	\[
		\# \mathrm{PF}_{gs - 1 \mid S} = s^{gs - 2}
	\]
	where $S$ is the set of the first $s$ natural numbers $j$ with $j \equiv 1 \pmod g$.
\end{restatable}

This is really the first case of a more general recursive result for parking functions of length $gs - k$. Here compositions $\lambda = (\lambda_{1}, \dots, \lambda_{n})$ with $\sum_{i = 1}^{n} \lambda_{i} = s$ are denoted $\lambda\vDash s, \lvert\lambda\rvert = n$ while $\binom{n}{\lambda}$ denotes the number of ways to partition a set of size $n$ into parts with sizes given by $\lambda_i$.

\begin{restatable}{theorem}{modPFcountGen}
	Given $g,s\in\mathbb{Z}_{>0},$ $1\le k\le gs,$ the number of parking functions of length $gs - k$ with gap $g$ between possible preferred spots is $\mathrm{PF}_{gs - k \mid S \cap [gs - k]}$ and satisfies the relation
	\[
		s^{gs - k} = s \, \# \mathrm{PF}_{gs - k \mid S \cap [gs - k]} + \sum_{n = 2}^{k} \frac{s}{n} \sum_{\substack{\lambda \vDash k \\ \lvert\lambda\rvert = n}} \sum_{\substack{\mu \vDash s \\\lvert \mu \rvert = n}} \binom{gs - k}{g \mu - \lambda} \prod_{i = 1}^{n} \# \mathrm{PF}_{g \mu_{i} - \lambda_{i} \mid S \cap [g \mu_{i} - \lambda_{i}]}
	\]
	where $S$ is the set of all natural numbers $j$ with $j \equiv 1 \pmod g$ and $\lambda \vDash k$ indicates that $\lambda$ is a composition of $k$ with length $|\lambda|.$
\end{restatable}

Note that here, for notational convenience, $S$ is not a subset of the various $m$ that are the lengths of the parking functions. Rather, $[m] \cap S$ is the relevant subset. 

\begin{proof}
    Consider $gs-k$ cars parking along a circular street with $gs$ spots; similarly to the proof of \cref{thm:modPFcount}, we only permit cars to prefer the $s$ distinct spots $1,g+1,2g+1,\ldots$ around the street. There are $s^{gs-k}$ total ways to choose these preferences, and every such choice will lead to all cars parking.

    The $k$ empty spots on the street will be distributed in $n$ contiguous segments for some $1\le n\le k.$ These unfilled segments will all be immediately prior to spots that can be preferred; all other scenarios force cars to pass unfilled spots without parking in them.

%    \begin{figure}[H]
%    	\centering
%	\includegraphics[width = 0.45 \textwidth]{figures/modPFs3}
%	\includegraphics[width = 0.45 \textwidth]{figures/modPFs4}
%	\caption{For $g = s = 3$ and $k = 2$, the two empty spots could be the last two spots $g \sigma + 1, g \sigma + 2$ in a ``row'', or the last spots $g \sigma_{1} + 2$ and $g \sigma_{2} + 2$ in two different ``rows''.}
%    \end{figure}

\begin{figure}[h]
    \centering
    \begin{tikzpicture}[scale=0.28]
        % parameters for spots and lot
        \def\lotor{7}
        \def\outerr{(\lotor - .5)};
        \def\innerr{(\lotor - 1.7)};
        \def\anglediff{30};
        \definecolor{carcolor}{RGB}{37, 51, 186}
        \newcommand{\circlespot}[2][40]{
            \begin{scope}[rotate around = {#1:(0,0)}]
                \draw[fill = black, rounded corners = 0.5ex] 
                        ({-\outerr*sin(\anglediff/2)},{\outerr*cos(\anglediff/2)}) 
                        arc (90+\anglediff/2:90-\anglediff/2:{\outerr})
                        -- ({\innerr*sin(\anglediff/2)},{\innerr*cos(\anglediff/2)}) arc (90-\anglediff/2:90+\anglediff/2:{\innerr}) -- cycle;
                \draw (0, {\outerr*.5+\innerr*.5}) node[text = white, rotate = #1]{#2};
            \end{scope}
        }

        % left half of diagram: missing cars together
        \begin{scope}     
        % parking Lot
        \draw[fill = black, fill opacity = .2, draw = none] (0,0) circle[radius = \lotor];        
        \draw[fill = white, draw = none] (0,0) circle[radius = \lotor-2.2];
        \draw[fill = none, dotted, thick] (0,0) circle[radius = \lotor + 3.5];

        % dotted lines
        \draw[dotted, thick] (0,0) -- ({(\lotor+3.5)*sin(    -20)}, {(\lotor+3.5)*cos(    -20)}); 
        \draw[dotted, thick] (0,0) -- ({(\lotor+3.5)*sin(40*3-20)}, {(\lotor+3.5)*cos(40*3-20)}); 
        \draw[   very thick] (0,0) -- ({(\lotor+4.5)*sin(40*6-20)}, {(\lotor+4.5)*cos(40*6-20)}); 
        \draw[->,very thick] ({(\lotor+4.5)*sin(40*6-20)}, {(\lotor+4.5)*cos(40*6-20)}) arc ({40*6-10}:{40*6-30}:{\lotor+4.5}); 
        \draw[fill = white, dotted, thick] (0,0) circle[radius = \lotor-2.7];

        % arrows
        \draw[->, thick] (0,0) -- (0,\lotor-3);                        
        \draw[->, thick] (0,0) -- ({(\lotor-3)*sin(40*3)}, {(\lotor-3)*cos(40*3)}); 
        \draw[->, thick] (0,0) -- ({(\lotor-3)*sin(40*6)}, {(\lotor-3)*cos(40*6)}); 
        \draw[fill = white, thin, align = center] (0,0) circle[radius = 2.5] node{\small Spots\\$g\sigma_i+1$};
        
        % parking spaces
        \circlespot[-40*0]{1}
        \circlespot[-40*1]{2}
        \circlespot[-40*2]{3}
        \circlespot[-40*3]{4}
        \circlespot[-40*4]{5}
        \circlespot[-40*5]{6}
        \circlespot[-40*6]{7}
        \circlespot[-40*7]{8}
        \circlespot[-40*8]{9}

        % cars
        \begin{scope}[rotate around = {-40*0:(0,0)}] \mzncarrot[(-2,\lotor)]{carcolor}{2}{-40*0} \end{scope}
        \begin{scope}[rotate around = {-40*1:(0,0)}] \mzncarrot[(-2,\lotor)]{carcolor}{3}{-40*1} \end{scope}
        \begin{scope}[rotate around = {-40*2:(0,0)}] \mzncarrot[(-2,\lotor)]{carcolor}{6}{-40*2} \end{scope}
        \begin{scope}[rotate around = {-40*3:(0,0)}] \mzncarrot[(-2,\lotor)]{carcolor}{7}{-40*3} \end{scope}
      % \begin{scope}[rotate around = {-40*4:(0,0)}] \mzncarrot[(-2,\lotor)]{carcolor}{0}{-40*4} \end{scope}
      % \begin{scope}[rotate around = {-40*5:(0,0)}] \mzncarrot[(-2,\lotor)]{carcolor}{0}{-40*5} \end{scope}
        \begin{scope}[rotate around = {-40*6:(0,0)}] \mzncarrot[(-2,\lotor)]{carcolor}{1}{-40*6} \end{scope}
        \begin{scope}[rotate around = {-40*7:(0,0)}] \mzncarrot[(-2,\lotor)]{carcolor}{4}{-40*7} \end{scope}
        \begin{scope}[rotate around = {-40*8:(0,0)}] \mzncarrot[(-2,\lotor)]{carcolor}{5}{-40*8} \end{scope}

        % text
        \draw[align = center] (0,-\lotor-6) node {Circular preference list: $(7,1,1,7,7,7,4)$.\\Linear result: $(1,4,4,1,1,7,1)\in \mathrm{PF}_{7 \mid \{ 1, 4, 7 \}}$.};
        \end{scope}

        % right half of diagram: missing cars apart
        \begin{scope}[shift = {(\lotor*2 + 13, 0)}]    
        % parking Lot
        \draw[fill = black, fill opacity = .2, draw = none] (0,0) circle[radius = \lotor];        
        \draw[fill = white, draw = none] (0,0) circle[radius = \lotor-2.2];
        \draw[fill = none, dotted, thick] (0,0) circle[radius = \lotor + 3.5];

        % dotted lines
        \draw[dotted, thick] (0,0) -- ({(\lotor+3.5)*sin(    -20)}, {(\lotor+3.5)*cos(    -20)}); 
        \draw[dotted, thick] (0,0) -- ({(\lotor+3.5)*sin(40*3-20)}, {(\lotor+3.5)*cos(40*3-20)}); 
        \draw[dotted, thick] (0,0) -- ({(\lotor+3.5)*sin(40*6-20)}, {(\lotor+3.5)*cos(40*6-20)}); 
        \draw[fill = white, dotted, thick] (0,0) circle[radius = \lotor-2.7];

        % arrows
        \draw[->, thick] (0,0) -- (0,\lotor-3);                        
        \draw[->, thick] (0,0) -- ({(\lotor-3)*sin(40*3)}, {(\lotor-3)*cos(40*3)}); 
        \draw[->, thick] (0,0) -- ({(\lotor-3)*sin(40*6)}, {(\lotor-3)*cos(40*6)}); 
        \draw[fill = white, thin, align = center] (0,0) circle[radius = 2.5] node{\small Spots\\$g\sigma_i+1$};
        
        % parking spaces
        \circlespot[-40*0]{1}
        \circlespot[-40*1]{2}
        \circlespot[-40*2]{3}
        \circlespot[-40*3]{4}
        \circlespot[-40*4]{5}
        \circlespot[-40*5]{6}
        \circlespot[-40*6]{7}
        \circlespot[-40*7]{8}
        \circlespot[-40*8]{9}

        % cars
        \begin{scope}[rotate around = {-40*0:(0,0)}] \mzncarrot[(-2,\lotor)]{carcolor}{1}{-40*0} \end{scope}
        \begin{scope}[rotate around = {-40*1:(0,0)}] \mzncarrot[(-2,\lotor)]{carcolor}{3}{-40*1} \end{scope}
    %   \begin{scope}[rotate around = {-40*2:(0,0)}] \mzncarrot[(-2,\lotor)]{carcolor}{1}{-40*2} \end{scope}
        \begin{scope}[rotate around = {-40*3:(0,0)}] \mzncarrot[(-2,\lotor)]{carcolor}{2}{-40*3} \end{scope}
        \begin{scope}[rotate around = {-40*4:(0,0)}] \mzncarrot[(-2,\lotor)]{carcolor}{4}{-40*4} \end{scope}
        \begin{scope}[rotate around = {-40*5:(0,0)}] \mzncarrot[(-2,\lotor)]{carcolor}{6}{-40*5} \end{scope}
        \begin{scope}[rotate around = {-40*6:(0,0)}] \mzncarrot[(-2,\lotor)]{carcolor}{5}{-40*6} \end{scope}
        \begin{scope}[rotate around = {-40*7:(0,0)}] \mzncarrot[(-2,\lotor)]{carcolor}{7}{-40*7} \end{scope}
    %   \begin{scope}[rotate around = {-40*8:(0,0)}] \mzncarrot[(-2,\lotor)]{carcolor}{1}{-40*8} \end{scope}

        \draw[align = center] (0,-\lotor-6) node {Circular preference list: $(1,4,1,4,7,4,4)$.\\No resulting parking function.};
        \end{scope}
    \end{tikzpicture}
    
    \caption{For $g = s = 3$ and $k = 2$, the two empty spots could be the last two spots $g\sigma + 2$ and $g\sigma + 3$ in a ``row'', or the last two spots $g\sigma_1 + 3$ and $g\sigma_2 + 3$ in two different ``rows''. Radial arrows indicate spots $g\sigma_i + 1$ which may be in the initial circular preference list.}
    \label{fig:circular}
\end{figure}

    Our gaps will be of size $\lambda_1,\lambda_2,\ldots,\lambda_n,$ all positive and summing up to $k$; we denote this as $\lambda\vDash k,|\lambda|=n.$ These unfilled segments can each be paired with the filled segments immediately prior to them; by virtue of the locations of unfilled segments, these pairs will each have total length a multiple of $g.$ Call these $g\mu_1,g\mu_2,\ldots,g\mu_n$; since the sum of all of these is $gs,$ we have $\mu\vDash s,|\mu|=n.$

    For each such partition of this type, each filled segment will have $g\mu_i-\lambda_i$ cars; choosing which cars will have preferences in each segment can be done in $\binom{gs - k}{g \mu_1 - \lambda_1,g \mu_2 - \lambda_2,\ldots,g \mu_n - \lambda_n}=\binom{gs - k}{g \mu - \lambda}$ ways.
    
    Within each segment, we have $g\mu_i-\lambda_i$ cars which can only prefer spots which are $1$ mod $g$ and must park in the first $g\mu_i-\lambda_i$ spots without gaps. This is precisely a parking function with modular restrictions, and can be done in $\# \mathrm{PF}_{gs - k \mid S \cap [gs - k]}$ ways.

    Finally, we must choose how our parking functions are placed within the circle itself. There are $s$ ways to place the first possible preference in the first restricted parking function, and all placements are forced from there onward; however, for any given $n$ we are over-counting by a factor of $n$, since all cyclic permutations of a given $\lambda$ and $\mu$ are equivalent, but all are counted. Thus we multiply by $\frac{s}{n}$, for a total of 
    \[
	    s^{gs - k} = \sum_{n = 1}^{k} \frac{s}{n} \sum_{\substack{\lambda \vDash k \\ \lvert\lambda\rvert = n}} \sum_{\substack{\mu \vDash s \\\lvert \mu \rvert = n}} \binom{gs - k}{g \mu - \lambda} \prod_{i = 1}^{n} \# \mathrm{PF}_{g \mu_{i} - \lambda_{i} \mid S \cap [g \mu_{i} - \lambda_{i}]}.
    \]
    The formula above arises from pulling out the desirable $n = 1$ term.
\end{proof}

In practice, the formula is computationally intractable for large $g, s, k$. However, it can be useful to show special cases of small $k$ as in \cref{thm:modPFcount} and below.

\begin{restatable}{example}{modPFcount2}
	The number of parking functions of length $gs - 2$ cars with gap $g$ between possible preferred spots is
	\[
		\#\mathrm{PF}_{gs - 2 \mid S} = s^{gs - 3} - \frac{1}{2} \sum_{i = 1}^{s - 1} \binom{gs - 2}{gs - 1} i^{gi - 2} (s - i)^{g(s - i) - 2}
	\]
	where $S$ is the set of the first $s$ natural numbers $j$ with $j \equiv 1 \pmod g$.
\end{restatable}

\section{Connections and further avenues} \label{sec:connections}

We conclude with a brief exploration of particularly interesting connections and leads for the enterprising reader to explore further.

\subsection{$\mathbf{u}$-parking functions}

$S$-restricted parking functions are closely related to a certain class of vector (or $\mathbf{u}$-)parking functions, associated to vectors $\mathbf{u}=(u_1,u_2,u_3,\ldots,u_n)$ and defined by \cite{yan-survey-2015} as functions whose $i$th-smallest output is at most $u_i$. 

By the following claim, any set of $S$-restricted parking functions is ``naturally" in bijection with a corresponding set of $\mathbf{u}$-parking functions:

\begin{theorem}
    For any $S\subseteq[n],$ $S$-restricted parking functions are in bijection with the $\mathbf{u}$-parking functions associated to the $\mathbf{u}$ given by $u_i=|S\cap [i]|.$
\end{theorem}
\begin{proof}
    Letting $s=|S|,$ take the function $f:[n]\to [s]$ given by sending $i$ to $u_i$; note that this is order-preserving and sends the $j$th-smallest element of $S$ to $j$. We claim that the function $f_*: \pi\mapsto f\circ \pi$ is a bijection between $S$-restricted parking functions and $\mathbf{u}$-parking functions.

    Note first that $f_*$ is injective since $f$ is injective on $S$; it also takes functions with codomain $S$ to functions with codomain $[s].$
    
    $S$-restricted parking functions are those functions $[n]\to S$ whose $i$th-smallest preference is at most $i$ for all $i\in [n]$. Since $f$ is order-preserving, the image of this set under $f_*$ is the set of functions $[n]\to [s]$ whose $i$th-smallest preference is at most $f(i)=u_i.$

    But this is precisely the set of $\mathbf{u}$-parking functions, as desired! We therefore have a bijection.
\end{proof}

This means that most of our results about restricted parking functions can also be phrased as results about particular classes of $\mathbf{u}$-parking functions!

\subsection{Hyperplane arrangements}

In considering permutation-invariant parking functions, \cite{chen-2023} comes across a class of $\mathbf{u}$-parking functions equivalent to initial-segment restrictions. In the particular cases of $s=1,2,3$, \cref{thm:resPFcount1} gives $1$, $2^n-1$, and $3^n-2^n-n$ $[s]$-restricted parking functions. They note in particular that $s=3$ generates the sequence \cite[A001263]{oeis}, as there happen to be $3^n-2^n-n$ regions in a structure called the $C_{n}$-Shi hyperplane arrangement where $C_{n}$ is the cyclic graph with $n$ vertices.

The Shi hyperplane arrangement in $n$ dimensions consists of the hyperplanes $x_i-x_j=0$ and $x_i-x_j=1$ for all $i,j\in[n]$ where $i<j$. Famously, it has $(n+1)^{n-1}$ regions, which are in bijection with parking functions through maps given by \cite{athanasiadis-linusson-1999} among others. \cite{bennett-2024} constructs a subset of this arrangement assigned to any labeled graph $G$ on $n$ vertices called the $G$-Shi arrangement. The $G$-Shi arrangement consists of hyperplanes $x_i-x_j=0$ and $x_i-x_j=1$ for all $i,j\in[n]$ where $i<j$ and the edge $ij$ is in $G$. In particular $C_n$-Shi arrangement consists of the hyperplanes $x_i-x_{i+1}=0,1$ for $i\in[n-1]$ along with the hyperplanes $x_1-x_n=0,1$. The fact that this connects to preference-restricted parking functions seems more than coincidence!

It seems worth considering this construction for other subgraphs, or even for other hyperplane arrangements. The Linial arrangement (as described by e.g. \cite{stanley-hyperplanes-2007}) consists of the hyperplanes $x_i-x_j=1$ for all $i,j\in[n]$ where $i<j.$ One could define a $G$-Linial arrangement in much the same way, consisting of the hyperplanes $x_i-x_j=1$ for all $i,j\in[n]$ where $i<j$ and the edge $ij$ is in $G$. In the case of $G=C_n$, this yields $2^n-1$ regions, as all ``possible" regions exist (except for $x_i-x_{i+1}>1$ for $i\in[n-1],$ $x_1-x_n<1$). This is exactly the number of $[2]$-restricted parking functions of length $n$. Cleverer constructions may yield interesting results.

\subsection{Parking polytopes}

In \cite{andres-others-2023}, the convex hull $\mathfrak{X}$ of weakly increasing $\mathbf{x}$-parking functions is studied in the special case $\mathbf{x} = (a, b, b, \dots, b)$ (recall that an $\mathbf{x}$-parking function is a $\mathbf{u}$-parking function with $u_{i} = x_{1} + \dots + x_{i}$). In particular, it is shown that $\mathfrak{X}$ is integrally equivalent to the Pitman-Stanley polytope of $\mathbf{x}$. This result provides a way to count integer points in the Pitman-Stanley polytope by counting weakly increasing parking functions and vice versa.

Our $[s]$-restricted parking functions correspond to $\mathbf{x}$-parking functions with $\mathbf{x} = (\underbrace{1, \dots, 1}_{s \text{ times}}, 0, \dots, 0)$. If the results of \cite{andres-others-2023} can be generalized to work with this $\mathbf{x}$, then we would have new ways to count integer points in the corresponding Pitman-Stanley polytope using our formulas for weakly increasing $[s]$-restricted parking functions. Given the Pitman-Stanley polytope's central importance, this introduces the possibility of connections to all kinds of geometric combinatorics.

\bibliography{references}
\bibliographystyle{alpha}

\end{document}